\documentclass[a4paper,oneside,english]{amsart}
\usepackage[T1]{fontenc}
\usepackage[latin9]{inputenc}
\usepackage{mathrsfs}
\usepackage{amstext}
\usepackage{amsthm}
\usepackage{amssymb}
\usepackage{stmaryrd}
\usepackage[all]{xy}

\makeatletter

\pdfpageheight\paperheight
\pdfpagewidth\paperwidth

\numberwithin{equation}{section}
\numberwithin{figure}{section}
\theoremstyle{plain}
\newtheorem{thm}{\protect\theoremname}
  \theoremstyle{definition}
  \newtheorem{defn}[thm]{\protect\definitionname}
  \theoremstyle{remark}
  \newtheorem{rem}[thm]{\protect\remarkname}
  \theoremstyle{plain}
  \newtheorem{prop}[thm]{\protect\propositionname}
  \theoremstyle{plain}
  \newtheorem{lem}[thm]{\protect\lemmaname}

\usepackage[dvipsnames,svgnames,x11names,hyperref]{xcolor}
\usepackage[colorlinks=true,linkcolor=Dark Red, citecolor=Dark Green,linktoc=all]{hyperref}
\usepackage{lmodern}
\renewcommand*{\epsilon}{\varepsilon}
\usepackage{amssymb}
\usepackage{amsfonts}
\usepackage{egothic}
\usepackage[T1]{fontenc}
\usepackage{url}
\usepackage{amsthm}
\usepackage{aliascnt}
\usepackage{lipsum}
\usepackage{mathrsfs}
\usepackage{esint}
\usepackage{url}
\usepackage{amsmath}
\usepackage{dsfont}
\usepackage{bbm}
\usepackage{pdflscape}
\usepackage{bbm}
\usepackage{bussproofs}
\newlength{\lhs} 
\newlength{\rhs} 

\makeatother

\usepackage{babel}
  \providecommand{\definitionname}{Definition}
  \providecommand{\lemmaname}{Lemma}
  \providecommand{\propositionname}{Proposition}
  \providecommand{\remarkname}{Remark}
\providecommand{\theoremname}{Theorem}

\begin{document}

\title{Yoneda Structures and KZ Doctrines}
\begin{abstract}
In this paper we strengthen the relationship between Yoneda structures
and KZ doctrines by showing that for any locally fully faithful KZ
doctrine, with the notion of admissibility as defined by Bunge and
Funk, all of the Yoneda structure axioms apart from the right ideal
property are automatic.
\end{abstract}

\author{Charles Walker}

\address{Department of Mathematics, Macquarie University, NSW 2109, Australia}

\email{charles.walker1@mq.edu.au}


\maketitle

\section{Introduction}

The majority of this paper concerns Kock-Z\"{o}berlein doctrines,
which were introduced by Kock \cite{kock1972} and Z\"{o}berlein
\cite{zober1976}. These KZ doctrines capture the free cocompletion
under a suitable class of colimits $\Phi$, with a canonical example
being the free small cocompletion KZ doctrine on locally small categories.
On the other hand, Yoneda structures as introduced by Street and Walters
\cite{yonedastructures} capture the presheaf construction, with the
canonical example being the Yoneda structure on (not necessarily locally
small) categories, whose basic data is the Yoneda embedding $\mathcal{A}\to\left[\mathcal{A}^{\textnormal{op}},\mathbf{Set}\right]$
for each locally small category $\mathcal{A}$. When $\mathcal{A}$
is small this coincides with the usual free small cocompletion, but
not in general. In this paper we prove a theorem tightening the relationship
between these two notions, not just in the context of this example,
but in general.

A key feature of a Yoneda structure (which is not present in the definition
of a KZ doctrine) is a class of 1-cells called \emph{admissible 1-cells}.
In the setting of the usual Yoneda structure on $\mathbf{CAT}$, a
1-cell (that is a functor) $L\colon\mathcal{A}\to\mathcal{B}$ is
called admissible when the corresponding functor $\mathcal{B}\left(L-,-\right)\colon\mathcal{B}\to\left[\mathcal{A}^{\textnormal{op}},\mathbf{SET}\right]$
factors through the inclusion of $\left[\mathcal{A}^{\textnormal{op}},\mbox{\ensuremath{\mathbf{Set}}}\right]$
into $\left[\mathcal{A}^{\textnormal{op}},\mathbf{SET}\right]$.

In order to compare Yoneda structures with KZ doctrines, we will also
need a notion of admissibility in the setting of a KZ doctrine. Fortunately,
such a notion of admissibility has already been introduced by Bunge
and Funk \cite{bungefunk}. In the case of the free small cocompletion
KZ doctrine $P$ on locally small categories, these admissible 1-cells,
which we refer to as $P$-admissible, are those functors $L\colon\mathcal{A}\to\mathcal{B}$
for which the corresponding functor $\mathcal{B}\left(L-,-\right)\colon\mathcal{B}\to\left[\mathcal{A}^{\textnormal{op}},\mbox{\ensuremath{\mathbf{Set}}}\right]$
factors through the inclusion of $P\mathcal{A}$ into $\left[\mathcal{A}^{\textnormal{op}},\mbox{\ensuremath{\mathbf{Set}}}\right]$.

The main result of this paper; Theorem \ref{admabll}, shows that
given a locally fully faithful KZ doctrine $P$ on a 2-category $\mathscr{C}$,
on defining the admissible maps to be those of Bunge and Funk, one
defines all the data and axioms for a Yoneda structure except for
the ``right ideal property'' which asks that the class of admissible
1-cells $\mathbf{I}$ satisfies the property that for each $L\in\mathbf{I}$
we have $L\cdot F\in\mathbf{I}$ for all $F$ such that the composite
$L\cdot F$ is defined.

\section{Background\label{background}}

In this section we will recall the notion of a KZ doctrine $P$ as
well as the notions of left extensions and left liftings, as these
will be needed to describe Yoneda structures, and to discuss their
relationship with KZ doctrines.
\begin{defn}
Suppose we are given a 2-cell $\eta\colon I\to R\cdot L$ as in the
left diagram
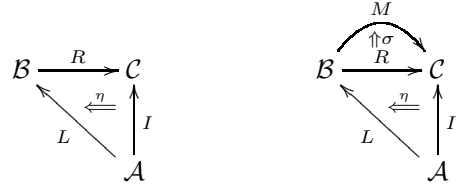
\[
\xymatrix@=1em{\mathcal{B}\ar[rr]^{R} &  & \mathcal{C}\ar@{}[ld]|-{\stackrel{\eta}{\Longleftarrow}} &  &  &  & \mathcal{B}\ar[rr]^{R}\ar@/^{1.5pc}/[rr]_{\Uparrow\sigma}^{M} &  & \mathcal{C}\ar@{}[ld]|-{\stackrel{\eta}{\Longleftarrow}}\\
 & \; &  &  &  &  &  & \;\\
 &  & \mathcal{A}\ar[uu]_{I}\ar[uull]^{L} &  &  &  &  &  & \mathcal{A}\ar[uu]_{I}\ar[uull]^{L}
}
\]
in a 2-category $\mathscr{C}$. We say that $R$ is exhibited as a
\emph{left extension} of $I$ along $L$ by the 2-cell $\eta$ when
pasting 2-cells $\sigma:R\to M$ with the 2-cell $\eta:I\to R\cdot L$
as in the right diagram defines a bijection between 2-cells $R\to M$
and 2-cells $I\to M\cdot L$. Moreover, we say such a left extension
is \emph{respected} by a 1-cell $E\colon\mathcal{C}\to\mathcal{D}$
when the whiskering of $\eta$ by $E$ given by the following pasting
diagram 
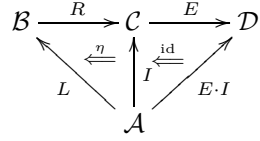
\[
\xymatrix@=1em{\mathcal{B}\ar[rr]^{R} &  & \mathcal{C}\ar@{}[ld]|-{\stackrel{\eta}{\Longleftarrow}}\ar[rr]^{E}\ar@{}[rd]|-{\stackrel{\textnormal{id}}{\Longleftarrow}} &  & \mathcal{D}\\
 & \; &  & \;\\
 &  & \mathcal{A}\ar[uu]_{I}\ar[uull]^{L}\ar[uurr]_{E\cdot I}
}
\]
exhibits $E\cdot R$ as a left extension of $E\cdot I$ along $L$. 

Dually, we have the notion of a left lifting. We say a 2-cell $\eta\colon I\to R\cdot L$
exhibits $L$ as a \emph{left lifting} of $I$ through $R$ when pasting
2-cells $\delta\colon L\to K$ with the 2-cell $\eta\colon I\to R\cdot L$
defines a bijection between 2-cells $L\to K$ and 2-cells $I\to R\cdot K$.
We call such a lifting \emph{absolute }if for any 1-cell $F\colon\mathcal{X}\to\mathcal{A}$
the whiskering of $\eta$ by $F$ given by the following pasting diagram
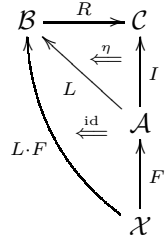
\[
\xymatrix@=1em{\mathcal{B}\ar[rr]^{R} &  & \mathcal{C}\ar@{}[ld]|-{\stackrel{\eta}{\Longleftarrow}}\\
 & \;\\
\ar@{}[rr]|-{\quad\stackrel{\textnormal{id}}{\Longleftarrow}} &  & \mathcal{A}\ar[uu]_{I}\ar[uull]^{L}\\
\\
 &  & \mathcal{X}\ar[uu]_{F}\ar@/^{1pc}/[uuuull]^{L\cdot F}
}
\]
exhibits $L\cdot F$ as a left lifting of $I\cdot F$ through $R$.
\end{defn}
There are quite a few different characterizations of KZ doctrines,
for example those due to Kelly-Lack or Kock \cite{lack1997,kock1972}.
For the purposes of relating KZ doctrines to Yoneda structures, it
will be easiest to work with the following characterization given
by Marmolejo and Wood \cite{marm2012} in terms of left Kan extensions.
\begin{defn}
\label{defkzdoctrine}\cite[Definition 3.1]{marm2012} A \emph{KZ
doctrine $\left(P,y\right)$ }on a 2-category $\mathscr{C}$ consists
of 

(i) An assignation on objects $P\colon\textnormal{ob}\mathscr{C}\to\textnormal{ob}\mathscr{C}$;

(ii) For every object $\mathcal{A}\in\mathscr{C}$, a 1-cell $y_{\mathcal{A}}\colon\mathcal{A}\to P\mathcal{A}$;

(iii) For every pair of objects $\mathcal{A}\text{ and }\mathcal{B}$
and 1-cell $F\colon\mathcal{A}\to P\mathcal{B}$, a left extension
\begin{equation}
\xymatrix@=1em{P\mathcal{A}\ar@{->}[rr]^{\overline{F}}\ar@{}[rd]|-{\stackrel{c_{F}}{\Longleftarrow}} &  & P\mathcal{B}\\
 & \;\\
\mathcal{A}\ar[rruu]_{F}\ar[uu]^{y_{\mathcal{A}}}
}
\label{diag:2.1}
\end{equation}
of $F$ along $y_{\mathcal{A}}$ exhibited by an isomorphism $c_{F}$
as above. 

Moreover, we require that:

(a) For every object $\mathcal{A}\in\mathscr{C}$, the left extension
of $y_{\mathcal{A}}$ as in \ref{diag:2.1} is given by
\[
\xymatrix@=1em{P\mathcal{A}\ar@{->}[rr]^{\textnormal{id}_{P\mathcal{A}}} &  & P\mathcal{A}\ar@{}[ld]|-{\stackrel{\textnormal{id}}{\Longleftarrow}}\\
 & \;\\
 &  & \mathcal{A}\ar[uu]_{y_{\mathcal{A}}}\ar[ulul]^{y_{\mathcal{A}}}
}
\]

Note that this means $c_{y_{\mathcal{A}}}$ is equal to the identity
2-cell on $y_{\mathcal{A}}$.

(b) For any 1-cell $G\colon\mathcal{B}\to P\mathcal{C}$, the corresponding
left extension $\overline{G}\colon P\mathcal{B}\to P\mathcal{C}$
respects the left extension $\overline{F}$ in \ref{diag:2.1}.\end{defn}
\begin{rem}
This definition is equivalent (in the sense that each gives rise to
the other) to the well known algebraic definition, which we refer
to as a KZ pseudomonad \cite{marm2012,marm1997}. A \emph{KZ pseudomonad
$\left(P,y,\mu\right)$ }on a 2-category $\mathscr{C}$ is taken to
be a pseudomonad $\left(P,y,\mu\right)$ on $\mathscr{C}$ equipped
with a modification $\theta\colon Py\to yP$ satisfying two coherence
axioms \cite{kock1972}.
\end{rem}
Just as KZ doctrines may be defined algebraically or in terms of left
extensions, one may also define pseudo algebras for these KZ doctrines
algebraically or in terms of left extensions. 

The following definitions in terms of left extensions are equivalent
to the usual notions of pseudo $P$-algebra and $P$-homomorphism,
in the sense that we have an equivalence between the two resulting
2-categories of pseudo $P$-algebras arising from the two different
definitions \cite[Theorems 5.1,5.2]{marm2012}.
\begin{defn}
[\cite{marm2012}] Given a KZ doctrine $\left(P,y\right)$ on a 2-category
$\mathscr{C}$, we say an object $\mathcal{X}\in\mathscr{C}$ is \emph{$P$-cocomplete}
if for every $G\colon\mathcal{B}\to\mathcal{X}$ 
\[
\xymatrix@=1em{P\mathcal{B}\ar@{->}[rr]^{\overline{G}}\ar@{}[rd]|-{\stackrel{c_{G}}{\Longleftarrow}} &  & \mathcal{X} &  &  & P\mathcal{A}\ar@{->}[rr]^{\overline{F}}\ar@{}[rd]|-{\stackrel{c_{F}}{\Longleftarrow}} &  & P\mathcal{B}\ar[rr]^{\overline{G}} &  & \mathcal{X}\\
 & \; &  &  &  &  & \;\\
\mathcal{B}\ar[rruu]_{G}\ar[uu]^{y_{\mathcal{B}}} &  &  &  &  & \mathcal{A}\ar[rruu]_{F}\ar[uu]^{y_{\mathcal{A}}}
}
\]
there exists a left extension $\overline{G}$ as on the left exhibited
by an isomorphism $c_{G}$, and moreover this left extension respects
the left extensions $\overline{F}$ as in the diagram on the right.
We say a 1-cell $E\colon\mathcal{X}\to\mathcal{Y}$ between $P$-cocomplete
objects $\mathcal{X}$ and $\mathcal{Y}$ is a \emph{$P$-homomorphism
}when it respects all left extensions along $y_{\mathcal{B}}$ into
$\mathcal{X}$ for every object $\mathcal{B}$. \end{defn}
\begin{rem}
It is clear that $P\mathcal{A}$ is $P$-cocomplete for every $\mathcal{A}\in\mathscr{C}$. 
\end{rem}
The relationship between $P$-cocompleteness and admitting a pseudo
$P$-algebra structure is as below. 
\begin{prop}
\label{Pcocompleteequiv}Given a KZ doctrine $\left(P,y\right)$ on
a 2-category $\mathscr{C}$ and an object $\mathcal{X}\in\mathscr{C}$,
the following are equivalent:

(1) $\mathcal{X}$ is $P$-cocomplete;

(2) $y_{\mathcal{X}}\colon\mathcal{X}\to P\mathcal{X}$ has a left
adjoint with invertible counit;

(3) $\mathcal{X}$ is the underlying object of a pseudo $P$-algebra.\end{prop}
\begin{proof}
For $\left(1\right)\iff\left(2\right)$ see the proof of \cite[Theorem 5.1]{marm2012},
and for $\left(2\right)\iff\left(3\right)$ see \cite{lack1997}.
\end{proof}
We now recall the notion of Yoneda structure as introduced by Street
and Walters \cite{yonedastructures}.
\begin{defn}
A \emph{Yoneda structure} $\mathfrak{Y}$ on a 2-category $\mathscr{C}$
consists of:

(1) A class of 1-cells $\mathbf{I}$ with the property that for any
$L\in\mathbf{I}$ we have $L\cdot F\in\mathbf{I}$ for all $F$ such
that the composite $L\cdot F$ is defined; we call this the class
of admissible 1-cells. We say an object $\mathcal{A}\in\mathscr{C}$
is admissible when $\textnormal{id}_{\mathcal{A}}$ is an admissible
1-cell. 

(2) For each admissible object $\mathcal{A}\in\mathscr{C}$, an admissible
map $y_{\mathcal{A}}\colon\mathcal{A}\to P\mathcal{A}$. 

(3) For each $L\colon\mathcal{A}\to\mathcal{B}$ such that $L$ and
$\mathcal{A}$ are both admissible, a 1-cell $R_{L}$ and 2-cell $\varphi_{L}$
as in the diagram
\[
\xymatrix@=1em{\mathcal{B}\ar[rr]^{R_{L}} &  & P\mathcal{A}\ar@{}[ld]|-{\stackrel{\varphi_{L}}{\Longleftarrow}}\\
 & \;\\
 &  & \mathcal{A}\ar[uu]_{y_{\mathcal{A}}}\ar[uull]^{L}
}
\]
Such that:

(a) The diagram above exhibits $L$ as a absolute left lifting and
$R_{L}$ as a left extension via $\varphi_{L}$.

(b) For each admissible $\mathcal{A}$, the diagram 
\[
\xymatrix@=1em{P\mathcal{A}\ar[rr]^{\textnormal{id}_{P\mathcal{A}}} &  & P\mathcal{A}\ar@{}[ld]|-{\stackrel{\textnormal{id }}{\Longleftarrow}}\\
 & \;\\
 &  & \mathcal{A}\ar[uu]_{y_{\mathcal{A}}}\ar[uull]^{y_{\mathcal{A}}}
}
\]
exhibits $\textnormal{id}_{P\mathcal{A}}$ as a left extension.

(c) For admissible $\mathcal{A},\mathcal{B}$ and $L,K$ as below,
the diagram 
\[
\xymatrix@=1em{P\mathcal{A} &  & P\mathcal{B}\ar@{}[ldld]|-{\stackrel{\varphi_{y_{\mathcal{B}\cdot}L}}{\implies}}\ar[ll]_{R_{y_{\mathcal{B}}\cdot L}}\ar@{}[rd]|-{\stackrel{\varphi_{K}}{\implies}} &  & \mathcal{C}\ar[ll]_{R_{K}}\\
 & \; &  & \;\\
\mathcal{A}\ar[uu]^{y_{\mathcal{A}}}\ar[rr]_{L} &  & \mathcal{B}\ar[uu]_{y_{\mathcal{B}}}\ar[urur]_{K}
}
\]
exhibits $R_{y_{\mathcal{B}}\cdot L}\cdot R_{K}$ as a left extension.\end{defn}
\begin{rem}
We note that when the admissible maps form a right ideal, the admissibility
of $L$ in condition (c) is redundant. However, in the following sections
we will consider a setting in which the admissible maps are closed
under composition, but do not necessarily form a right ideal.
\end{rem}

\begin{rem}
There is an additional axiom (d) discussed in ``Yoneda structures''
\cite{yonedastructures} which when satisfied defines a so called
\emph{good} Yoneda structure \cite{weberyoneda}. This axiom asks\emph{
}for every\emph{ }admissible $L$ and every diagram
\[
\xymatrix@=1em{\mathcal{B}\ar@{->}[rr]^{M} &  & P\mathcal{A}\ar@{}[ld]|-{\stackrel{\phi}{\Longleftarrow}}\\
 & \;\\
 &  & \mathcal{A}\ar[uu]_{y_{\mathcal{A}}}\ar[ulul]^{L}
}
\]
that if $\phi$ exhibits $L$ as an absolute left lifting, then $\phi$
exhibits $M$ as a left extension. This condition implies axioms (b)
and (c) in the presence of (a) \cite[Prop. 11]{yonedastructures}. 

However, this condition is often too strong. For example one may consider
the free $\mathbf{Cat}$-cocompletion, and take $\mathbb{N}$ to be
the monoid of natural numbers seen as a one object category, yielding
the absolute left lifting diagram
\[
\xymatrix@=1em{\mathbf{1}\ar@{->}[rr]^{\textnormal{pick }\mathbb{N}} &  & \mathbf{Cat}\ar@{}[ld]|-{\stackrel{!}{\Longleftarrow}}\\
 & \;\\
 &  & \mathbf{1}\ar[uu]_{\textnormal{pick }\mathbf{1}}\ar[ulul]^{\textnormal{id}_{\mathbf{1}}}
}
\]
It is then trivial, as we would be extending along an identity, that
the left extension property is not satisfied. 
\end{rem}

\section{Admissible Maps in KZ Doctrines}

Yoneda structures as defined above require us to give a suitable class
of admissible maps, and so in order to compare Yoneda structures with
KZ doctrines we will need a suitable notion of admissible map in the
setting of a KZ doctrine. Bunge and Funk defined a map $L\colon\mathcal{A}\to\mathcal{B}$
in the setting of a KZ pseudomonad $P$ to be $P$-admissible when
$PL$ has a right adjoint, and showed this notion of admissibility
may also be described in terms of left extensions \cite{bungefunk}.
Our definition in terms of left extensions and KZ doctrines is as
follows.
\begin{defn}
\label{defadm} Given a KZ doctrine $\left(P,y\right)$ on a 2-category
$\mathscr{C}$, we say a 1-cell $L\colon\mathcal{A}\to\mathcal{B}$
is \emph{$P$-admissible} when 
\[
\xymatrix@=1em{\mathcal{B}\ar@{->}[rr]^{R_{L}} &  & P\mathcal{A}\ar@{}[ld]|-{\stackrel{\varphi_{L}}{\Longleftarrow}} &  &  & \mathcal{B}\ar@{->}[rr]^{R_{L}} &  & P\mathcal{A}\ar@{}[ld]|-{\stackrel{\varphi_{L}}{\Longleftarrow}}\ar[rr]^{\overline{H}}\ar@{}[rd]|-{\stackrel{c_{H}}{\Longleftarrow}} &  & \mathcal{X}\\
 & \; &  &  &  &  & \; &  & \;\\
 &  & \mathcal{A}\ar[uu]_{y_{\mathcal{A}}}\ar[ulul]^{L} &  &  &  &  & \mathcal{A}\ar[uu]_{y_{\mathcal{A}}}\ar[ulul]^{L}\ar[rruu]_{H}
}
\]
 there exists a left extension $\left(R_{L},\varphi_{L}\right)$ of
$y_{\mathcal{A}}$ along $L$ as in the left diagram, and moreover
the left extension is respected by any $\overline{H}$ as in the right
diagram where $\mathcal{X}$ is $P$-cocomplete.\end{defn}
\begin{rem}
Note that such a $\overline{H}$ is a $P$-homomorphism, and conversely
that a $P$-homomorphism $\overline{H}\colon P\mathcal{A}\to\mathcal{X}$
is a left extension of $H:=\overline{H}\cdot y_{\mathcal{A}}$ along
$y_{\mathcal{A}}$ as above. Thus this is saying the left extension
$R_{L}$ is respected by $P$-homomorphisms.\end{rem}
\begin{lem}
\label{leftadj} Suppose we are given a KZ doctrine $\left(P,y\right)$
and a \emph{$P$-}admissible 1-cell $L\colon\mathcal{A}\to\mathcal{B}$
where $\mathcal{B}$ is $P$-cocomplete, then the 1-cell $R_{L}$
in 
\[
\xymatrix@=1em{\mathcal{B}\ar@{->}[rr]^{R_{L}} &  & P\mathcal{A}\ar@{}[ld]|-{\stackrel{\varphi_{L}}{\Longleftarrow}}\\
 & \;\\
 &  & \mathcal{A}\ar[uu]_{y_{\mathcal{A}}}\ar[ulul]^{L}
}
\]
has a left adjoint $\overline{L}\colon P\mathcal{A}\to\mathcal{B}$.\end{lem}
\begin{proof}
Taking $\overline{L}$ to be the left extension 
\[
\xymatrix@=1em{P\mathcal{\mathcal{A}}\ar@{}[rd]|-{\stackrel{c_{L}}{\Longleftarrow}}\ar@{->}[rr]^{\overline{L}} &  & \mathcal{B}\\
 & \;\\
\mathcal{A}\ar[rruu]_{L}\ar[uu]^{y_{\mathcal{A}}}
}
\]
we then have $\overline{L}\dashv R_{L}$ since we may define $n:\textnormal{id}_{P\mathcal{A}}\to R_{L}\cdot\overline{L}$
and $e:\overline{L}\cdot R_{L}\to\textnormal{id}_{\mathcal{B}}$ respectively
as (since $L$ is\emph{ $P$-}admissible) the unique solutions to
\[
\xymatrix@=1em{ & \mathcal{B}\ar[rd]^{R_{L}} &  &  &  &  &  &  &  & \;\\
P\mathcal{A}\ar[ru]^{\overline{L}}\ar[rr]_{\textnormal{id}_{P\mathcal{A}}} & \ar@{}[u]|-{\Uparrow n} & P\mathcal{A} & \; & P\mathcal{A}\ar[r]^{\overline{L}} & \mathcal{B}\ar[r]^{R_{L}} & P\mathcal{A} &  & \mathcal{B}\ar[r]^{R_{L}}\ar@/^{2pc}/[rr]^{\textnormal{id}_{\mathcal{B}}} & P\mathcal{A}\ar[r]^{\overline{L}}\ar@{}[u]|-{\Uparrow e} & \mathcal{B} & \ar@{}[d]|-{=} & \mathcal{B}\ar[rr]^{\textnormal{id}_{\mathcal{B}}} &  & \mathcal{B}\\
 & \ar@{}[ur]|-{\quad\stackrel{\textnormal{id}}{\Rightarrow}} &  & \ar@{}[u]|-{=} &  & \;\ar@{}[ru]|-{\;\quad\stackrel{\varphi_{L}}{\Leftarrow}}\ar@{}[u]|-{\stackrel{c_{L}}{\Leftarrow}\quad} & \; &  &  & \;\ar@{}[ru]|-{\quad\;\stackrel{c_{L}}{\Leftarrow}}\ar@{}[u]|-{\stackrel{\varphi_{L}}{\Leftarrow}\quad} & \; & \; &  & \;\ar@{}[ur]|-{\stackrel{\textnormal{id}}{\Leftarrow}}\\
 &  & \mathcal{A}\ar[ulul]^{y_{\mathcal{A}}}\ar[uu]_{y_{\mathcal{A}}} &  &  &  & \mathcal{A}\ar[uul]^{L}\ar[uu]_{y_{\mathcal{A}}}\ar@/^{1pc}/[ulul]^{y_{\mathcal{A}}} &  &  &  & \mathcal{A}\ar[uul]^{y_{\mathcal{A}}}\ar[uu]_{L}\ar@/^{1pc}/[ulul]^{L} &  &  &  & \mathcal{A}\ar[uu]_{L}\ar[uull]^{L}
}
\]
Verifying the triangle identities is then a simple exercise.
\end{proof}
The following is an easy consequence of this Lemma.
\begin{lem}
\label{lanresran} Suppose we are given a KZ doctrine $\left(P,y\right)$
on a 2-category $\mathscr{C}$ and a \emph{$P$-}admissible 1-cell
$L\colon\mathcal{A}\to\mathcal{B}$. Then the 1-cell $\textnormal{res}_{L}$
defined here as the left extension in the top triangle
\[
\xymatrix@=1em{P\mathcal{A} &  & P\mathcal{B}\ar[ll]_{\textnormal{res}_{L}}\\
 & \;\ar@{}[ur]|-{\stackrel{c_{R_{L}}}{\implies}}\\
\mathcal{A}\ar[uu]^{y_{\mathcal{A}}}\ar[rr]_{L}\ar@{}[ur]|-{\stackrel{\varphi_{L}}{\implies}\quad} &  & \mathcal{B}\ar[uu]_{y_{\mathcal{B}}}\ar[ulul]^{R_{L}}
}
\]
has a left adjoint $\textnormal{lan}_{L}$, and when $R_{L}$ is \emph{$P$-}admissible,
a right adjoint $\textnormal{ran}_{L}$.\end{lem}
\begin{proof}
First note that it is an easy consequence of the left extension pasting
lemma (the dual of \cite[Prop. 1]{yonedastructures}) that $y_{\mathcal{B}}\cdot L$
is $P$-admissible, which is to say the left extension $\textnormal{res}_{L}$
above is respected by any $P$-homomorphism $\overline{H}\colon P\mathcal{A}\to\mathcal{X}$.
This is since such a $\overline{H}$ will respect the left extension
$R_{L}$ of $y_{\mathcal{A}}$ along $L$ as well as the left extension
$\textnormal{res}_{L}$ of $R_{L}$ along $y_{\mathcal{B}}$. Hence
by Lemma \ref{leftadj} $\textnormal{res}_{L}$ has a left adjoint
$\textnormal{lan}_{L}$ given as the left extension as on the left
(which is how $PL$ is defined given the data of Definition \ref{defkzdoctrine}),
\[
\xymatrix@=1em{P\mathcal{A}\ar@{->}[rr]^{\textnormal{lan}_{L}}\ar@{}[rdrd]|-{\stackrel{c_{y_{\mathcal{B}\cdot L}}}{\Longleftarrow}} &  & P\mathcal{B} &  &  & P\mathcal{A}\ar[rr]^{\textnormal{ran}_{L}} &  & P\mathcal{B}\\
 &  &  &  &  &  & \;\ar@{}[ur]|-{\stackrel{\varphi_{R_{L}}}{\implies}}\\
\mathcal{A}\ar[rr]_{L}\ar[uu]^{y_{\mathcal{A}}} &  & \mathcal{B}\ar[uu]_{y_{\mathcal{B}}} &  &  &  &  & \mathcal{B}\ar[uu]_{y_{\mathcal{B}}}\ar[ulul]^{R_{L}}
}
\]
and if $R_{L}$ is \emph{$P$-}admissible then we may define $\textnormal{ran}_{L}:=R_{R_{L}}$
(which is the left extension as on the right) and since $P\mathcal{A}$
is $P$-cocomplete $\textnormal{ran}_{L}$ has a left adjoint given
by $\textnormal{res}_{L}=\overline{R_{L}}$ again by Lemma \ref{leftadj}.\end{proof}
\begin{rem}
We have shown that when both $L$ and $R_{L}$ are $P$-admissible
we have the adjoint triple $PL\dashv\overline{R_{L}}\dashv R_{R_{L}}$.
Of particular interest is the case where $L=y_{\mathcal{A}}$ for
some $\mathcal{A}\in\mathscr{C}$. Clearly in this case both $L$
and $R_{L}$ are $P$-admissible and so we may define $\mu_{\mathcal{A}}:=\overline{R_{y_{\mathcal{A}}}}=\overline{\textnormal{id}_{P\mathcal{A}}}$
and observe $R_{R_{y_{\mathcal{A}}}}=R_{\textnormal{id}_{P\mathcal{A}}}=y_{P\mathcal{A}}$
to recover the well known sequence of adjunctions $Py_{\mathcal{A}}\dashv\mu_{\mathcal{A}}\dashv y_{P\mathcal{A}}$
as in \cite{marm1997}.
\end{rem}
The following result is mostly due to Bunge and Funk \cite{bungefunk},
though we state it in our notation and from the viewpoint of KZ doctrines
in terms of left extensions. Also, we will prove the following proposition
in full detail in order to clarify some parts of the argument given
by Bunge and Funk \cite{bungefunk}. For example, in order to check
that certain left extensions are respected we will need to know their
exhibiting 2-cells. These exhibiting 2-cells will also be needed later
to prove our main result.
\begin{prop}
\label{admequiv} Given a KZ doctrine $\left(P,y\right)$ on a 2-category
$\mathscr{C}$ and a 1-cell $L\colon\mathcal{A}\to\mathcal{B}$, the
following are equivalent:

(1) $L$ is \emph{$P$-}admissible;

(2) every $P$-cocomplete object $\mathcal{X}\in\mathscr{C}$ admits,
and $P$-homomorphism respects, left extensions along $L$. This says
that for any given 1-cell $K\colon\mathcal{A}\to\mathcal{X}$, where
$\mathcal{X}$ is $P$-cocomplete, there exists a 1-cell $J$ and
2-cell $\delta$ as on the left 
\[
\xymatrix@=1em{\mathcal{B}\ar@{->}[rr]^{J} &  & \mathcal{X}\ar@{}[ld]|-{\stackrel{\delta}{\Longleftarrow}} &  &  & \mathcal{B}\ar@{->}[rr]^{J} &  & \mathcal{X}\ar@{}[ld]|-{\stackrel{\delta}{\Longleftarrow}}\ar[rr]^{E} &  & \mathcal{Y}\\
 & \; &  &  &  &  & \; &  & \;\\
 &  & \mathcal{A}\ar[uu]_{K}\ar[ulul]^{L} &  &  &  &  & \mathcal{A}\ar[uu]_{K}\ar[ulul]^{L}
}
\]
exhibiting $J$ as a left extension, and moreover this left extension
is respected by any $P$-homomorphism $E\colon\mathcal{X}\to\mathcal{Y}$
for $P$-cocomplete $\mathcal{Y}$ as in the right diagram.

(3) $PL:=\textnormal{lan}_{L}$ given as the left extension
\[
\xymatrix@=1em{P\mathcal{A}\ar@{->}[rr]^{PL}\ar@{}[rdrd]|-{\stackrel{c_{y_{\mathcal{B}\cdot L}}}{\Longleftarrow}} &  & P\mathcal{B}\\
\\
\mathcal{A}\ar[rr]_{L}\ar[uu]^{y_{\mathcal{A}}} &  & \mathcal{B}\ar[uu]_{y_{\mathcal{B}}}
}
\]
has a right adjoint. \emph{We denote the inverse of the above 2-cell
as $y_{L}:=c_{y_{\mathcal{B}\cdot L}}^{-1}$ for every 1-cell $L$.}\end{prop}
\begin{proof}
The following implications prove the logical equivalence.

$\left(2\right)\implies\left(1\right)\colon$ This is trivial as $P\mathcal{A}$
is $P$-cocomplete.

$\left(1\right)\implies\left(2\right)\colon$ Given a $K\colon\mathcal{A}\to\mathcal{X}$
as in (2). We take the pasting 
\[
\xymatrix@=1em{\mathcal{B}\ar@{->}[rr]^{R_{L}} &  & P\mathcal{A}\ar@{}[ld]|-{\stackrel{\varphi_{L}}{\Longleftarrow}}\ar[rr]^{\overline{K}}\ar@{}[rd]|-{\stackrel{c_{K}}{\Longleftarrow}} &  & \mathcal{X}\\
 & \; &  & \;\\
 &  & \mathcal{A}\ar[uu]_{y_{\mathcal{A}}}\ar[ulul]^{L}\ar[rruu]_{K}
}
\]
as our left extension using that $L$ is \emph{$P$-}admissible. This
is respected by any $P$-homomorphism $E\colon\mathcal{X}\to\mathcal{Y}$
where $\mathcal{Y}$ is $P$-cocomplete as a consequence of the second
part of the definition of $P$-admissibility.

$\left(1\right)\implies\left(3\right)\colon$ This was shown in Lemma
\ref{lanresran}.

$\left(3\right)\implies\left(1\right)\colon$ This implication is
where the majority of the work lies in proving this proposition. We
suppose that we are given an adjunction $\textnormal{lan}_{L}\dashv\textnormal{res}_{L}$
with unit $\eta$ where $\textnormal{lan}_{L}$ is defined as in (3).
We split the proof into two parts.

\textsc{Part 1:} \emph{The given right adjoint, $\textnormal{res}_{L}$,
is a left extension of $\textnormal{res}_{L}\cdot y_{\mathcal{B}}$
along $y_{\mathcal{B}}$ as in the diagram} 
\[
\xymatrix@=1em{P\mathcal{B}\ar[rrrr]^{\textnormal{res}_{L}}\ar@{}[rrd]|-{\stackrel{\textnormal{id}}{\Longleftarrow}} &  &  &  & P\mathcal{A}\\
 &  & P\mathcal{B}\ar[rru]_{\textnormal{res}_{L}}\\
\mathcal{B}\ar[uu]^{y_{\mathcal{B}}}\ar[rru]_{y_{\mathcal{B}}}
}
\]
\emph{exhibited by the identity 2-cell.}\footnote{This may be seen as an analogue of \cite[Prop. 1.3]{bungefunk}. However, we emphasize here that considering right adjoints tells us $\textnormal{res}_{L}$ is a $P$-homomorphism since the adjunctions may be used to construct an isomorphism between $\textnormal{res}_{L}$ and a known $P$-homomorphism. }

To see this, we consider the isomorphism in the square on the left
\[
\xymatrix@=1em{P\mathcal{A}\ar@{->}[rr]^{PL}\ar@{}[rdrd]|-{\stackrel{y_{L}}{\implies}} &  & P\mathcal{B} &  &  & P^{2}\mathcal{A}\ar@{->}[rr]^{P^{2}L}\ar@{}[rdrd]|-{\stackrel{Py_{L}}{\implies}} &  & P^{2}\mathcal{B} &  &  & P^{2}\mathcal{A}\ar@{}[rdrd]|-{\cong}\ar[dd]_{\mu_{\mathcal{A}}} &  & P^{2}\mathcal{B}\ar[ll]_{P\textnormal{res}_{L}}\ar[dd]^{\mu_{\mathcal{B}}}\\
\\
\mathcal{A}\ar[rr]_{L}\ar[uu]^{y_{\mathcal{A}}} &  & \mathcal{B}\ar[uu]_{y_{\mathcal{B}}} &  &  & P\mathcal{A}\ar[rr]_{PL}\ar[uu]^{Py_{\mathcal{A}}} &  & P\mathcal{B}\ar[uu]_{Py_{\mathcal{B}}} &  &  & P\mathcal{A} &  & P\mathcal{B}\ar[ll]^{\textnormal{res}_{L}}
}
\]
and then apply $P$ to get the isomorphism of left adjoints in the
middle square (suppressing pseudofunctoriality constraints\footnote{These pseudofunctoriality constraints are those arising from the uniqueness of left extensions up to coherent isomorphism.}),
which corresponds to an isomorphism of right adjoints in the right
square (which we leave unnamed). Now by \cite[Theorem 4.2]{marm2012}
(and since $\mu_{\mathcal{A}}\cdot P\textnormal{res}_{L}$ respects
the left extension $Py_{\mathcal{B}}$) we have the left extension
$\mu_{\mathcal{A}}\cdot P\textnormal{res}_{L}\cdot Py_{\mathcal{B}}$
of $\textnormal{res}_{L}\cdot y_{\mathcal{B}}$ along $y_{\mathcal{B}}$
as below
\[
\xymatrix@=1em{ &  & \ar@{}[d]|-{\cong} &  & P\mathcal{B}\ar[rrd]^{\textnormal{res}_{L}}\ar@{}[d]|-{\cong}\\
P\mathcal{B}\ar[rr]^{Py_{\mathcal{B}}}\ar@/^{0.75pc}/[rrrru]^{\textnormal{id}_{P\mathcal{B}}} &  & P^{2}\mathcal{B}\ar[rr]^{P\textnormal{res}_{L}}\ar@/^{0.25pc}/[rru]^{\mu_{\mathcal{B}}} &  & P^{2}\mathcal{A}\ar[rr]^{\mu_{\mathcal{A}}}\ar@{}[rd]|-{\cong} &  & P\mathcal{A}\\
 & \ar@{}[]|-{\Downarrow y_{y_{\mathcal{B}}}} &  & \ar@{}[]|-{\quad\Downarrow y_{\textnormal{res}_{L}}} &  & \;\\
\mathcal{B}\ar[rr]_{y_{\mathcal{B}}}\ar[uu]^{y_{\mathcal{B}}} &  & P\mathcal{B}\ar[rr]_{\textnormal{res}_{L}}\ar[uu]_{y_{P\mathcal{B}}} &  & P\mathcal{A}\ar[rruu]_{\textnormal{id}_{P\mathcal{A}}}\ar[uu]_{y_{P\mathcal{A}}}
}
\]
and so pasting with the isomorphism $\mu_{\mathcal{A}}\cdot P\textnormal{res}_{L}\cdot Py_{\mathcal{B}}\cong\textnormal{res}_{L}$
constructed as above tells us $\textnormal{res}_{L}$ is also an extension
of $\textnormal{res}_{L}\cdot y_{\mathcal{B}}$ along $y_{\mathcal{B}}$.
It follows that $\textnormal{res}_{L}$ respects the left extension
\[
\xymatrix@=1em{P\mathcal{B}\ar[rr]^{\textnormal{id}_{P\mathcal{B}}}\ar@{}[rd]|-{\stackrel{\textnormal{id}}{\Longleftarrow}} &  & P\mathcal{B}\\
 & \;\\
\mathcal{B}\ar[uu]^{y_{\mathcal{B}}}\ar[uurr]_{y_{\mathcal{B}}}
}
\]
and this gives the result. 

\textsc{Part 2:}\emph{ The following pasting exhibits} 
\[
\xymatrix@=1em{\mathcal{B}\ar[rr]^{y_{\mathcal{B}}} &  & P\mathcal{B}\ar[rr]^{\textnormal{res}_{L}} &  & P\mathcal{A}\ar[rr]^{\overline{H}} &  & \mathcal{X}\\
 & \; & \ar@{}[l]|-{\;\;\;\stackrel{y_{L}}{\Longleftarrow}} & P\mathcal{A}\ar[ul]^{\textnormal{lan}_{L}}\ar@{}[u]|-{\stackrel{\eta}{\Longleftarrow}}\ar[ru]_{\textnormal{id}_{P\mathcal{A}}} & \ar@{}[0,1]|-{\stackrel{c_{H}}{\Longleftarrow}\;\;\;} & \;\\
 &  &  & \mathcal{A}\ar@/^{1pc}/[ulull]^{L}\ar[u]^{y_{\mathcal{A}}}\ar@/_{1pc}/[rurru]_{H} & \;
}
\]
\emph{the composite $\overline{H}\cdot\textnormal{res}_{L}\cdot y_{\mathcal{B}}$
as a left extension of $H$ along $L$. }

Suppose we are given a 1-cell $K\colon\mathcal{B}\to\mathcal{X}$.
We then see that our left extension is exhibited by the sequence of
natural bijections\def\fCenter{\ \rightarrow\ } 
\def\ScoreOverhang{2pt}
\settowidth{\rhs}{$\overline{K}\cdot\textnormal{lan}_{L}\cdot y_{\mathcal{A}}$} 
\settowidth{\lhs}{$\overline{H}\cdot\textnormal{res}_{L}\cdot y_{\mathcal{B}}$}
\begin{prooftree}
\Axiom$\makebox[\lhs][r]{$H$} \fCenter \makebox[\rhs][l]{$K \cdot L$}$
\RightLabel{$\qquad\overline{K}\cdot y_{\mathcal{B}}\cong K$} 
\UnaryInf$\makebox[\lhs][r]{$H$} \fCenter \makebox[\rhs][l]{$\overline{K} \cdot y_\mathcal{B} \cdot L$}$ 
\RightLabel{$\qquad\textnormal{lan}_{L}\cdot y_{\mathcal{A}}\cong y_{\mathcal{B}}\cdot L$}
\UnaryInf$\makebox[\lhs][r]{$H$} \fCenter \makebox[\rhs][l]{$\overline{K}\cdot\textnormal{lan}_{L}\cdot y_{\mathcal{A}}$}$ 
\RightLabel{$\qquad c_{H}\textnormal{ exhibits }\overline{H}\textnormal{ as a left extension}$}
\UnaryInf$\makebox[\lhs][r]{$\overline H$} \fCenter \makebox[\rhs][l]{$\overline K \cdot \textnormal{lan}_L$}$ 
\RightLabel{$\qquad \textnormal{mates correspondence}$}
\UnaryInf$\makebox[\lhs][r]{$\overline{H}\cdot\textnormal{res}_{L}$} \fCenter \makebox[\rhs][l]{$\overline{K}$}$ 
\RightLabel{$\qquad \textnormal{left extension \ensuremath{\textnormal{res}_{L}} in Part 1 preserved by }\overline{H}$}
\UnaryInf$\makebox[\lhs][r]{$\overline{H}\cdot\textnormal{res}_{L}\cdot y_{\mathcal{B}}$} \fCenter \makebox[\rhs][l]{$\overline{K}\cdot y_\mathcal{B}$}$ 
\RightLabel{$\qquad \overline{K}\cdot y_{\mathcal{B}}\cong K$}
\UnaryInf$\makebox[\lhs][r]{$\overline{H}\cdot\textnormal{res}_{L}\cdot y_{\mathcal{B}}$} \fCenter \makebox[\rhs][l]{$K$}$ 
\end{prooftree}

It is easily seen this left extension is exhibited by the above 2-cell
since when taking $K=\overline{H}\cdot\textnormal{res}_{L}\cdot y_{\mathcal{B}}$
we may take $\overline{K}=\overline{H}\cdot\textnormal{res}_{L}$
as a consequence of Part 1 (with the left extension $\overline{K}$
exhibited by the identity 2-cell). Tracing through the bijection to
find the exhibiting 2-cell is then trivial. \end{proof}
\begin{rem}
\label{defvarphi} Considering Part 2 in the above proposition with
$H=y_{\mathcal{A}}$ and $\overline{H}$ and $c_{H}$ being an identity
1-cell and 2-cell respectively, we see that for any $P$-admissible
1-cell $L\colon\mathcal{A}\to\mathcal{B}$ and corresponding adjunction
$PL\dashv\textnormal{res}_{L}$ with unit $\eta$, we may define our
1-cell $R_{L}$ and 2-cell $\varphi_{L}$ as in Definition \ref{defadm}
by 
\[
\xymatrix@=1em{\mathcal{B}\ar@{->}[rr]^{R_{L}} &  & P\mathcal{A}\ar@{}[ld]|-{\stackrel{\varphi_{L}}{\Longleftarrow}} &  &  &  & \mathcal{B}\ar[rr]^{y_{\mathcal{B}}} & \; & P\mathcal{B}\ar[r]^{\textnormal{res}_{L}} & P\mathcal{A}\\
 & \; &  &  & := &  &  &  & \ar@{}[ul]|-{\stackrel{y_{L}}{\Longleftarrow}} & P\mathcal{A}\ar@/^{0.5pc}/[ul]^{\textnormal{lan}_{L}}\ar@{}[ul]|-{\quad\stackrel{\eta}{\Longleftarrow}}\ar[u]_{\textnormal{id}_{P\mathcal{A}}}\\
 &  & \mathcal{A}\ar[uu]_{y_{\mathcal{A}}}\ar[ulul]^{L} &  &  &  &  &  &  & \mathcal{A}\ar@/^{1pc}/[ulull]^{L}\ar[u]_{y_{\mathcal{A}}}
}
\]

We will make regular use of this definition in the next section.
\end{rem}

\begin{rem}
It is clear from the above proposition that \emph{$P$-}admissible
1-cells are closed under composition as noted by Bunge and Funk \cite{bungefunk}.
We may also note, as in \cite{bungefunk}, that every left adjoint
is $P$-admissible, as taking $PL:=\textnormal{lan}_{L}$ defines
a pseudofunctor \cite[Theorem 4.1]{marm2012} and so preserves the
adjunction.
\end{rem}

\section{Relating KZ doctrines and Yoneda Structures \label{yonedastructuresandkzdoctrines}}

We are now ready to prove our main result. In the following statement
we call a KZ doctrine locally fully faithful if the unit components
are fully faithful; indeed Bunge and Funk \cite{bungefunk} noted
that a KZ pseudomonad is locally fully faithful precisely when its
unit components  are fully faithful. Here the admissible maps of Bunge
and Funk refer to those maps $L$ for which $PL:=\textnormal{lan}_{L}$
has a right adjoint (which we denote by $\textnormal{res}_{L}$).
\begin{thm}
\label{admabll} Suppose we are given a locally fully faithful KZ
doctrine $\left(P,y\right)$ on a 2-category $\mathscr{C}$. Then
on defining the class of admissible maps $L$ to be those of Bunge
and Funk, with chosen left extensions $\left(R_{L},\varphi_{L}\right)$
those of Remark \ref{defvarphi}, we obtain all of the definition
and axioms of a Yoneda structure with the exception of the right ideal
property (though the admissible maps remain closed under composition).\end{thm}
\begin{proof}
We need only check that: 

\emph{(1) $\varphi_{L}$ exhibits $L$ as an absolute left lifting.}
Thus, we must exhibit a natural bijection between 2-cells $L\cdot W\to H$
and 2-cells $y_{\mathcal{A}}\cdot W\to R_{L}\cdot H$ for 1-cells
$W\colon\mathcal{D}\to\mathcal{A}$ and $H\colon\mathcal{D}\to\mathcal{B}$
as in the diagram
\[
\xymatrix@=1em{\mathcal{D}\ar[rr]^{W}\ar@/_{1pc}/[rdrd]_{H} & \; & \mathcal{A}\ar[rr]^{y_{\mathcal{A}}}\ar[dd]_{L} &  & P\mathcal{A}\\
 & \;\ar@{}[u]|-{\Downarrow\alpha} & \; & \;\ar@{}[ul]|-{\Downarrow\varphi_{L}}\\
 &  & \mathcal{B}\ar[dd]_{y_{\mathcal{B}}}\ar[ruru]_{R_{L}}\\
 &  &  & \;\ar@{}[ul]|-{\Downarrow c_{R_{L}}}\\
 &  & P\mathcal{B}\ar@/_{1pc}/[ruruuu]_{\textnormal{res}_{L}}
}
\]
Such a natural bijection is given by the correspondence\def\fCenter{\ \rightarrow\ } 
\def\ScoreOverhang{2pt}
\settowidth{\rhs}{$\textnormal{res}_{L}\cdot y_{\mathcal{B}}\cdot H$} 
\settowidth{\lhs}{$\textnormal{lan}_{L}\cdot y_{\mathcal{A}}\cdot W$}
\begin{prooftree}
\Axiom$\makebox[\lhs][r]{$L\cdot W$} \fCenter \makebox[\rhs][l]{$H$}$
\RightLabel{$\qquad y_{\mathcal{B}}\textnormal{ fully faithful}$} 
\UnaryInf$\makebox[\lhs][r]{$y_{\mathcal{B}}\cdot L\cdot W$} \fCenter \makebox[\rhs][l]{$y_{\mathcal{B}}\cdot H$}$ 
\RightLabel{$\qquad \textnormal{lan}_{L}\cdot y_{\mathcal{A}}\cong y_{\mathcal{B}}\cdot L$}
\UnaryInf$\makebox[\lhs][r]{$\textnormal{lan}_{L}\cdot y_{\mathcal{A}}\cdot W$} \fCenter \makebox[\rhs][l]{$y_{\mathcal{B}}\cdot H$}$ 
\RightLabel{$\qquad \textnormal{lan}_{L}\dashv\textnormal{res}_{L}$}
\UnaryInf$\makebox[\lhs][r]{$y_{\mathcal{A}}\cdot W$} \fCenter \makebox[\rhs][l]{$\textnormal{res}_{L}\cdot y_{\mathcal{B}}\cdot H$}$ 
\RightLabel{$\qquad R_{L}:=\textnormal{res}_{L}\cdot y_{\mathcal{B}}$}
\UnaryInf$\makebox[\lhs][r]{$y_{\mathcal{A}}\cdot W$} \fCenter \makebox[\rhs][l]{$R_{L}\cdot H$}$ 
\end{prooftree}and the 2-cell exhibiting this absolute left lifting is easily seen
to be the 2-cell as given in Remark \ref{defvarphi} by following
the above bijection.

\emph{(2) $\textnormal{res}_{L}\cdot R_{K}$ is a left extension.}
Considering the diagram
\[
\xymatrix@=1em{P\mathcal{A} &  & P\mathcal{B}\ar[ll]_{\textnormal{res}_{L}}\ar@{}[dr]|-{\stackrel{\varphi_{K}}{\implies}} &  & \mathcal{C}\ar[ll]_{R_{K}}\\
 & \;\ar@{}[ru]|-{\stackrel{c_{R_{L}}}{\implies}}\ar@{}[dl]|-{\stackrel{\varphi_{L}}{\implies}\quad} &  & \;\\
\mathcal{A}\ar[uu]^{y_{\mathcal{A}}}\ar[rr]_{L} &  & \mathcal{B}\ar[rruu]_{K}\ar[uu]_{y_{\mathcal{B}}}\ar[ulul]^{R_{L}}
}
\]
we first note that $\textnormal{res}_{L}\cdot R_{K}$ is a left extension
of $R_{L}$ along $K$ since $K$ is \emph{$P$-}admissible. We then
apply the pasting lemma for left extensions to see the outside diagram
also exhibits $\textnormal{res}_{L}\cdot R_{K}$ as a left extension. \end{proof}
\begin{rem}
We observe that to ask that $\textnormal{res}_{L}\cdot R_{K}$ be
a left extension in the diagram above for every $P$-admissible $L$
and $K$, is to ask by the pasting lemma that the pasting of $\varphi_{K}$
and $c_{R_{L}}$ exhibit $\textnormal{res}_{L}\cdot R_{K}$ as a left
extension. As $c_{R_{L}}$ is invertible, this is to say that $\textnormal{res}_{L}$
respects every left extension arising from admissibility. This is
equivalent to asking $\textnormal{res}_{L}$ be a $P$-homomorphism.
\end{rem}

\begin{rem}
We note here that we do not necessarily have the right ideal property.
Indeed given a KZ doctrine on a 2-category every identity arrow is
admissible, and so the right ideal property would require all arrows
into all objects being admissible (that is all arrows being admissible).
This fails for example with the identity KZ doctrine on any 2-category
$\mathscr{C}$ which contains an arrow $L$ with no right adjoint.
\end{rem}

\begin{rem}
Given an object $\mathcal{A}\in\mathscr{C}$ with a $P$-admissible
generalized element $a\colon\mathcal{S}\to\mathcal{A}$ we have a
version of the Yoneda lemma in the sense that we have bijections\def\fCenter{\ \rightarrow\ } 
\def\ScoreOverhang{2pt}
\settowidth{\rhs}{$\textnormal{res}_{a}\cdot K$} 
\settowidth{\lhs}{$\textnormal{lan}_{a}\cdot y_{\mathcal{S}}$}
\begin{prooftree}
\Axiom$\makebox[\lhs][r]{$y_{\mathcal{A}}\cdot a$} \fCenter \makebox[\rhs][l]{$K$}$
\RightLabel{$\qquad \textnormal{lan}_{a}\cdot y_{\mathcal{S}}\cong y_{\mathcal{A}}\cdot a$} 
\UnaryInf$\makebox[\lhs][r]{$\textnormal{lan}_{a}\cdot y_{\mathcal{S}}$} \fCenter \makebox[\rhs][l]{$K$}$
\RightLabel{$\qquad \textnormal{lan}_{a}\dashv\textnormal{res}_{a}$}
\UnaryInf$\makebox[\lhs][r]{$y_{\mathcal{S}}$} \fCenter \makebox[\rhs][l]{$\textnormal{res}_{a}\cdot K$}$ 
\end{prooftree}for generalized elements $K\colon\mathcal{S}\to P\mathcal{A}$. In
the case where $P$ is the usual free small cocompletion KZ doctrine
on locally small categories and $\mathcal{S}=\mathbf{1}$ is the terminal
category, maps $y_{\mathcal{S}}\to\textnormal{res}_{a}\cdot K$ are
elements of $\textnormal{res}_{a}\cdot K$ (which may be viewed as
$K$ evaluated at $a$). 
\end{rem}
The purpose of the following is to give an example in which absolute
left liftings (also known as relative adjunctions or partial adjunctions)
are preserved\footnote{In this case respected by the KZ pseudomonad resulting from the KZ doctrine as in \cite{marm2012}.}.
Also, the following proposition does not require locally fully faithfulness,
whereas Theorem \ref{admabll} does. 
\begin{prop}
Suppose we are given a KZ doctrine $\left(P,y\right)$ on a 2-category
$\mathscr{C}$. Then for every \emph{$P$-}admissible 1-cell $L\colon\mathcal{A}\to\mathcal{B}$
as on the left,
\[
\xymatrix@=1em{\mathcal{B}\ar@{->}[rr]^{R_{L}} &  & P\mathcal{A}\ar@{}[ld]|-{\stackrel{\varphi_{L}}{\Longleftarrow}} &  &  &  & P\mathcal{B}\ar@{->}[rr]^{PR_{L}} &  & P^{2}\mathcal{A}\ar@{}[ld]|-{\stackrel{P\varphi_{L}}{\Longleftarrow}}\\
 & \; &  &  &  &  &  & \;\\
 &  & \mathcal{A}\ar[uu]_{y_{\mathcal{A}}}\ar[ulul]^{L} &  &  &  &  &  & P\mathcal{A}\ar[uu]_{Py_{\mathcal{A}}}\ar[ulul]^{PL}
}
\]
the 2-cell $P\varphi_{L}$ as on the right (in which we have suppressed
the pseudofunctoriality constraints) exhibits $PL$ as an absolute
left lifting of $Py_{\mathcal{A}}$ through $PR_{L}$.\end{prop}
\begin{proof}
Without loss of generality, we define $\varphi_{L}$ as in Remark
\ref{defvarphi}. We then have the sequence of natural bijections\def\fCenter{\ \rightarrow\ } 
\def\ScoreOverhang{2pt}
\settowidth{\rhs}{$P\textnormal{res}_{L}\cdot Py_{\mathcal{B}}\cdot H$} 
\settowidth{\lhs}{$P^{2}L\cdot Py_{\mathcal{A}}\cdot W$}
\begin{prooftree}
\Axiom$\makebox[\lhs][r]{$PL\cdot W$} \fCenter \makebox[\rhs][l]{$H$}$
\RightLabel{$\qquad Py_{\mathcal{B}}\textnormal{ fully faithful}$} 
\UnaryInf$\makebox[\lhs][r]{$Py_{\mathcal{B}}\cdot PL\cdot W$} \fCenter \makebox[\rhs][l]{$Py_{\mathcal{B}}\cdot H$}$ 
\RightLabel{$\qquad y_{\mathcal{B}}\cdot L\cong PL\cdot y_{\mathcal{A}}$}
\UnaryInf$\makebox[\lhs][r]{$P^{2}L\cdot Py_{\mathcal{A}}\cdot W$} \fCenter \makebox[\rhs][l]{$Py_{\mathcal{B}}\cdot H$}$ 
\RightLabel{$\qquad PL\dashv\textnormal{res}_{L}$}
\UnaryInf$\makebox[\lhs][r]{$Py_{\mathcal{A}}\cdot W$} \fCenter \makebox[\rhs][l]{$P\textnormal{res}_{L}\cdot Py_{\mathcal{B}}\cdot H$}$ 
\RightLabel{$\qquad R_{L}:=\textnormal{res}_{L}\cdot y_{\mathcal{B}}$}
\UnaryInf$\makebox[\lhs][r]{$Py_{\mathcal{A}}\cdot W$} \fCenter \makebox[\rhs][l]{$PR_{L}\cdot H$}$ 
\end{prooftree}for 1-cells $W$ into $P\mathcal{A}$. Following the bijection we
see that the absolute left lifting is exhibited by $P\varphi_{L}$,
suppressing the pseudofunctoriality constraints.
\end{proof}
Some observations made in ``Yoneda structures'' \cite{yonedastructures}
may be seen more directly in this setting of a KZ doctrine. For example
Street and Walters defined an admissible morphism $L$ (in the setting
of a Yoneda structure) to be fully faithful when the 2-cell $\varphi_{L}$
is invertible (which agrees with a representable notion of fully faithfulness,
that is fully faithfulness defined via the absolute left lifting property,
when axiom (d) is satisfied). Here we see this in the context of a
(locally fully faithful) KZ doctrine.
\begin{prop}
\label{phiinv} Suppose we are given a KZ doctrine $\left(P,y\right)$
on a 2-category $\mathscr{C}$, and a \emph{$P$-}admissible 1-cell
$L\colon\mathcal{A}\to\mathcal{B}$ 
\[
\xymatrix@=1em{\mathcal{B}\ar@{->}[rr]^{R_{L}} &  & P\mathcal{A}\ar@{}[ld]|-{\stackrel{\varphi_{L}}{\Longleftarrow}}\\
 & \;\\
 &  & \mathcal{A}\ar[uu]_{y_{\mathcal{A}}}\ar[ulul]^{L}
}
\]
with a left extension $R_{L}$ as in the above diagram. Then the exhibiting
2-cell $\varphi_{L}$ is invertible if and only if $PL:=\textnormal{lan}_{L}$
is fully faithful. \end{prop}
\begin{proof}
We use the well known fact that the left adjoint of an adjunction
is fully faithful precisely when the unit is invertible. Now, given
that $\varphi_{L}$ is invertible we may define our 2-cell $\eta^{\ast}$
as the unique solution to 
\[
\xymatrix@=1em{ &  &  &  &  &  & \;\ar@{}[d]|-{\Uparrow\eta^{\ast}}\\
P\mathcal{A} &  & P\mathcal{A}\ar@{}[ld]|-{\stackrel{\varphi_{L}^{-1}}{\implies}}\ar@{->}[ll]_{\textnormal{id}_{P\mathcal{A}}} &  & P\mathcal{A} &  & P\mathcal{B}\ar@{}[rrdd]|-{\Uparrow c_{y_{\mathcal{B}}\cdot L}}\ar[ll]_{\textnormal{res}_{L}} &  & P\mathcal{A}\ar[ll]_{PL}\ar@/_{2pc}/[llll]_{\textnormal{id}_{P\mathcal{A}}}\\
 & \mathcal{B}\ar[ul]^{R_{L}} &  & = &  & \;\ar@{}[ur]|-{\Uparrow c_{R_{L}}}\\
 &  & \mathcal{A}\ar[uu]_{y_{\mathcal{A}}}\ar[ul]^{L} &  &  &  & \mathcal{B}\ar[uu]^{y_{\mathcal{B}}}\ar[ulul]^{R_{L}} &  & \mathcal{A}\ar[uu]_{y_{\mathcal{A}}}\ar[ll]^{L}
}
\]
That $\eta$ is the inverse of $\eta^{\ast}$ follows from an easy
calculation using Remark \ref{defvarphi}. Conversely, if the unit
$\eta$ is invertible then so is $\varphi_{L}$ by Remark \ref{defvarphi}.\end{proof}
\begin{rem}
If we define a map $L$ to be $P$-fully faithful when $PL$ is fully
faithful, then as a consequence of Proposition \ref{admequiv} (Part
2) and Proposition \ref{phiinv} we see that for any $P$-admissible
map $L$, this $L$ is $P$-fully faithful if and only if every left
extension along $L$ into a $P$-cocomplete object is exhibited by
an invertible 2-cell.
\end{rem}
In the following remark we compare $PL$ being fully faithful with
$L$ being fully faithful, and point out sufficient conditions for
these notions to agree.
\begin{rem}
Note that if $PL$ is fully faithful then $L$ is fully faithful assuming
$P$ is locally fully faithful, as $y$ is pseudonatural. Conversely
if $L$ is fully faithful, then (supposing our corresponding left
extension $R_{L}$ is pointwise) the exhibiting 2-cell is invertible
\cite[Prop. 2.22]{weberyoneda}, equivalent to $PL$ being fully faithful
by the above. This converse may also be seen when the KZ doctrine
is locally fully faithful and good (meaning axiom (d) is satisfied
for $P$-admissible maps) as we can use the argument of \cite[Prop. 9]{yonedastructures}.
However, as we now see, this converse need not hold in general.

An example in which $L$ is fully faithful but $PL$ is not is given
as follows. Take $\mathcal{A}$ to be the 2-category containing the
two objects $0,1$ and two non-trivial 1-cells $x,y\colon0\to1$,
and take $\mathcal{B}$ to be the same but with an additional 2-cell
$\alpha\colon x\to y$. Define $L$ as the inclusion of $\mathcal{A}$
into $\mathcal{B}$. Then for the free $\mathbf{Cat}$-cocompletion
of $\mathcal{A}$ given by $y_{\mathcal{A}}\colon\mathcal{A}\to\left[\mathcal{A}^{\textnormal{op}},\mathbf{Cat}\right]$
we note that $y_{\mathcal{A}}$ and $R_{L}\cdot L$ are not isomorphic,
and so the 2-cell $\varphi_{L}$ is not invertible meaning $PL$ is
not fully faithful (despite $L$ being fully faithful).
\end{rem}

\section{Future Work}

We have seen that the notions of pseudo algebras and admissibility
for a given KZ doctrine, and KZ doctrines themselves, may be expressed
in terms of left extensions. In a soon forthcoming paper we show that
pseudodistributive laws over a KZ doctrine may be simply expressed
entirely in terms of left extensions and admissibility, allowing us
to generalize some results of Marmolejo and Wood \cite{marm2012}.

\section{Acknowledgments}

The author would like to thank his supervisor as well as the anonymous
referee for their helpful feedback. In addition, the support of an
Australian Government Research Training Program Scholarship is gratefully
acknowledged.

\bibliographystyle{siam}
\bibliography{references}

\end{document}